\documentclass{amsart}

\usepackage{amssymb}
\usepackage{amsrefs}
\usepackage{overpic}
\usepackage{enumitem}   
\usepackage{graphicx} 
\usepackage[hidelinks]{hyperref}

\graphicspath{{Figures/}} 

\newtheorem{theorem}{Theorem}%[section]
\newtheorem{proposition}[theorem]{Proposition}
\newtheorem{definition}[theorem]{Definition}

\newtheorem{lemma}[theorem]{Lemma}
\newtheorem{example}[theorem]{Example}
\newtheorem{remark}[theorem]{Remark}

\begin{document}\title[Limit cycles and chaos in planar hybrid systems]{Limit cycles and chaos in planar hybrid systems}
	
\author[Jaume Llibre and Paulo Santana]
{Jaume Llibre$^1$ and Paulo Santana$^2$}
	
\address{$^1$ Departament de Matem\`{a}tiques, Facultat de Ci\`{e}ncies, Universitat Aut\`{o}noma de Barcelona, 08193 Bellaterra, Barcelona, Spain}
\email{jaume.llibre@uab.cat}
	
\address{$^2$ IBILCE-UNESP, CEP 15054-000, S. J. Rio Preto, S\~ao Paulo, Brazil}
\email{paulo.santana@unesp.br}
	
	
	
\begin{abstract}
	In this paper we study the family of planar hybrid differential systems formed by two linear centers and a polynomial reset map of any degree. We study their limit cycles and also provide examples of these hybrid systems exhibiting chaotic dynamics.
\end{abstract}
	
\maketitle
	
\section{Introduction}\label{Sec1}

An \emph{hybrid dynamical system} is a system whose dynamics are governed both by continuous and discrete laws. That is, a system that can both \emph{flow} and \emph{jump}. The field of hybrid dynamical systems is relatively new and its notion is very wide, encompassing a myriad of different definitions and formalisms, depending on which scientific community is approaching it. Following Schaft and Schumacher~\cite[Section~$1.1$]{SchSch2000}, we quote: 

\emph{[$\dots$] the area of hybrid systems is still in its infancy and a general theory of hybrid systems seems premature. More inherently, hybrid systems is such a wide notion that sticking to a single definition shall be too restrictive. [$\dots$] Another difficulty in discussing hybrid systems is that various scientific communities with their own approaches have contributed (and are still contributing) to the area. At least the following three communities can be distinguished. [The] computer science community [$\dots$] [the] modeling and simulation community [$\dots$] [and the] systems and control community}.

Therefore, in this paper it is not our goal to provide a complete introduction to the whole field of hybrid dynamical systems. For this, we refer to the books of Schaft and Schumacher~\cite{SchSch2000} and di Bernardo et al \cite{Bernardo2008}. Rather, our goal is to provide an example of approach in this area by applying the tools of \emph{Qualitative Theory of Differential Equations and Dynamical Systems}. One of the usual examples of an hybrid dynamical system is the \emph{bouncing ball model}

\begin{example}
	Consider the vertical motion of a ball dropped (or tossed) from an initial height $h>0$, with initial velocity $v\in\mathbb{R}$ (here negative velocity means downwards, following gravity, while positive velocity means that the ball was tossed upwards). If the ball is under the acceleration of constant gravity $g>0$, then for $h>0$ the state of the ball is governed by the system of differential equations,
	\begin{equation}\label{18}
		\dot h=v, \quad \dot v=-g.
	\end{equation}
	Therefore given any initial condition $q=(h_0,v_0)$, $h_0>0$, it follows from \eqref{18} that the ball reaches the ground $h=0$ after a finite amount of time $t_0>0$, with velocity $v(t_0)<0$. See Figure~\ref{Fig5}.
\begin{figure}[h]
	\begin{center}
		\begin{overpic}[height=5cm]{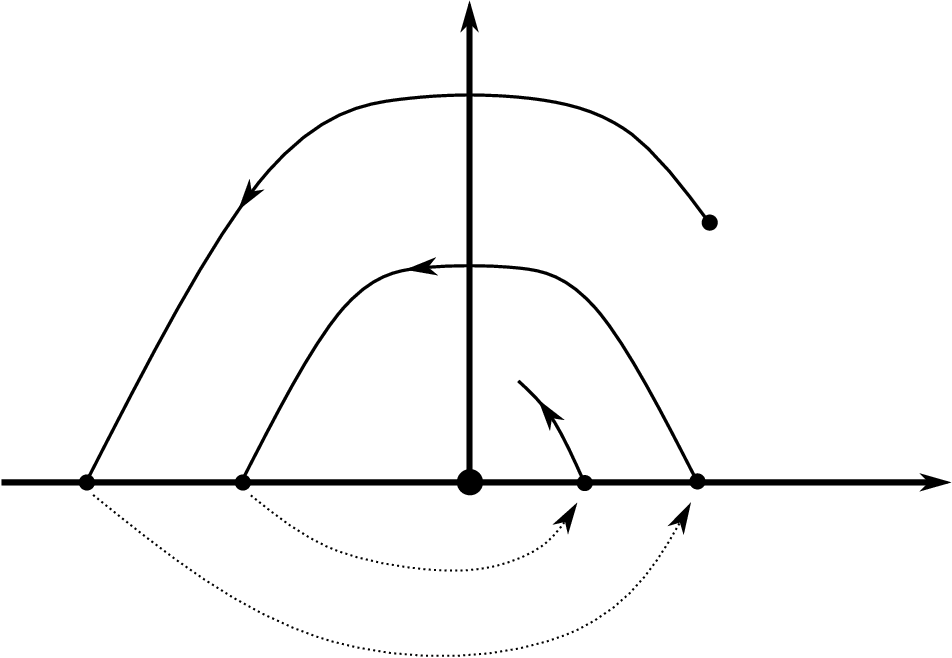} 
		%\begin{overpic}[height=5cm,grid,tics=5]{Fig9.eps} 
			\put(51,65){$h$}
			\put(97,20){$v$}
			\put(45,13){$\mathcal{O}$}
			\put(55,8){$\varphi$}
			\put(65,3){$\varphi$}
			\put(76,45){$q$}			
		\end{overpic}		
	\end{center}
\caption{Illustration of an orbit of the bouncing ball model. For simplicity, we interchanged the coordinate axes.}\label{Fig5}
\end{figure}
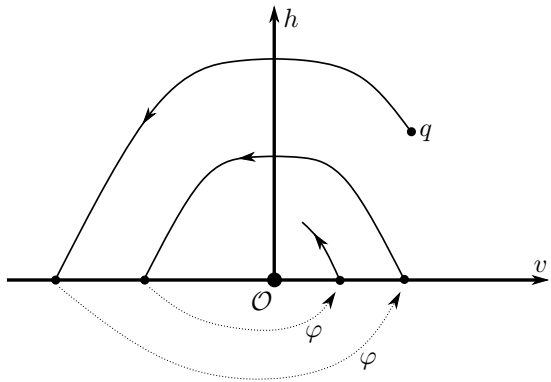
At this instant, the ball bounces back upwards and its velocity undergoes an instantaneous change $v(t_0)\mapsto -\rho v(t_0)$, modeled by an inelastic collision with dissipate factor $0<\rho<1$. That is, at $h=0$ our system undergoes a jump (or reset) $\varphi$ given by $\varphi(0,v)=(0,-\rho v)$. Roughly speaking, the reset map represents the instantaneous loss of energy that the system suffers when hitting the ground. Observe that any orbit converges to the origin, which represents a ball standing still in the ground.
\end{example}

For more details on the bouncing ball model, we refer to \cite[Section $2.2.3$]{SchSch2000} and \cite[Section $1.2$]{Bernardo2008}. Although the field of hybrid systems is relatively new, it has been receiving attention in very recent years with works on impact limit cycles \cite{LiLiu2023}, stability of hybrid polycycles \cite{SanSer2024}, Melnikov method to detect limit cycles \cite{LiWeiZhaKap2023} and on higher dimensions on the euclidian space \cite{WanZha2024}.

We now present the family of hybrid dynamical systems that we study in this paper. Let $\Sigma=\{(x,y)\in\mathbb{R}^2\colon x=0\}$ be a switching manifold. For simplicity we identify any map $\varphi\colon\Sigma\to\Sigma$ with its projection on the second coordinate. Let
\[
A_1=\{(x,y)\in\mathbb{R}^2\colon x<0\}, \quad A_2=\{(x,y)\in\mathbb{R}^2\colon x>0\},
\]
and let $\overline{A_i}$ denote the topological closure of $A_i$, $i\in\{1,2\}$. Of course, $A_1$ represents the left half flow and $A_2$ represents the right half flow. Given $n\geqslant1$, let $\mathfrak{X}_n$ be the class of the hybrid dynamical systems $X=(X_1,X_2;\Sigma;\varphi_n)$ such that $X_1$ and $X_2$ are linear vector fields with a center singularity, $\varphi_n\colon\Sigma\to\Sigma$ is a polynomial of degree $n$ and such that $X_i$ is defined over $\overline{A_i}$, $i\in\{1,2\}$. We explain the dynamics of $X\in\mathfrak{X}_n$ by defining a global orbit $\Gamma(q,k,t)$ of $X$ through an initial point $q\in\mathbb{R}^2$, with $k\in\mathbb{N}$, $t\geqslant0$ and $\Gamma(q,1,0)=q$.

Suppose first $q\in\mathbb{R}^2\backslash\Sigma$. In this case $\Gamma(q,k,t)$ is unique and works as follows. Without loss of generality suppose $q\in A_1$. Let $\gamma_1(t)$ be the orbit of $X_1$ with initial condition $q$. If $\gamma_1$ never intersects $\Sigma$, then $\Gamma(q,k,t)\equiv\Gamma(q,1,t)$ and $\Gamma(q,1,t)=\gamma_1(t)$ for every $t\geqslant0$. Suppose now that there is a $t_1>0$ such that $q_1=\gamma_1(t_1)\in\Sigma$ and $\gamma_1(t)\not\in\Sigma$ for all $t\in[0,t_1)$. After reaching $q_1$ the orbit \emph{jumps} to $q_2=\varphi_n(q_1)$. Let $\gamma_2(t)$ be the orbit of $X_2$ with initial condition $q_2$. Observe that exactly one of the following statements hold.
\begin{enumerate}[label=(\roman*)]
	\item $\gamma_2(\varepsilon)\in A_2$ for $\varepsilon>0$ sufficiently small.
	\item $\gamma_2(\varepsilon)\in\overline{A_1}$ for $\varepsilon>0$ sufficiently small.
\end{enumerate}
If $(i)$ holds, then we can enter in $A_2$ by following $\gamma_2(t)$ and thus the process repeats. That is, we follow $\gamma_2(t)$ until we reach $\Sigma$ again in a point $\overline{q_2}=\gamma_2(t_2)$, $t_2>0$, and then we jump to $q_3=\varphi_n(\overline{q_2})$, try to enter in $A_1$ and so on. See Figure~\ref{Fig6}.
\begin{figure}[h]
	\begin{center}
		\begin{minipage}{5cm}
			\begin{center}
				\begin{overpic}[height=5cm]{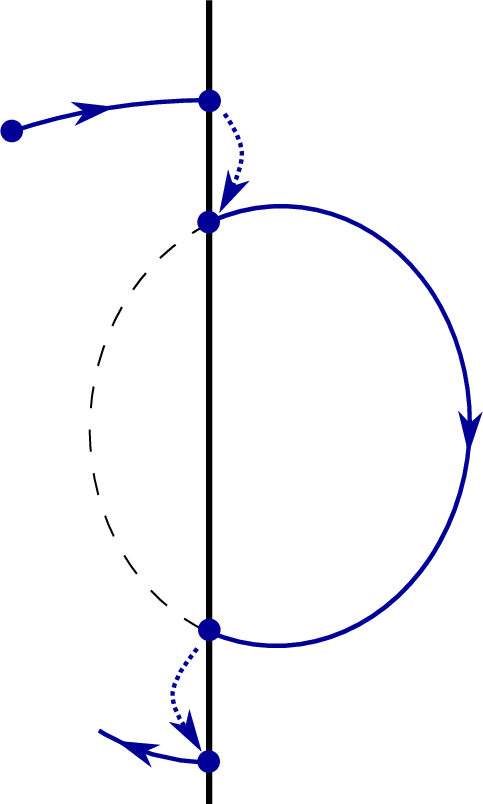} 
				%\begin{overpic}[height=5cm,grid,tics=5]{Fig12.eps} 
					\put(20,96){$\Sigma$}
					\put(0,78){$q$}
					\put(10,89){$\gamma_1$}
					\put(28,87){$q_1$}
					\put(31,80){$\varphi_n$}
					\put(28,68){$q_2$}
					\put(51,68){$\gamma_2$}
					\put(28,23){$\overline{q_2}$}
					\put(12,14){$\varphi_n$}
					\put(29,4){$q_3$}
				\end{overpic}
				
				$(a)$
			\end{center}
		\end{minipage}
		\begin{minipage}{5cm}
			\begin{center}
				\begin{overpic}[height=5cm]{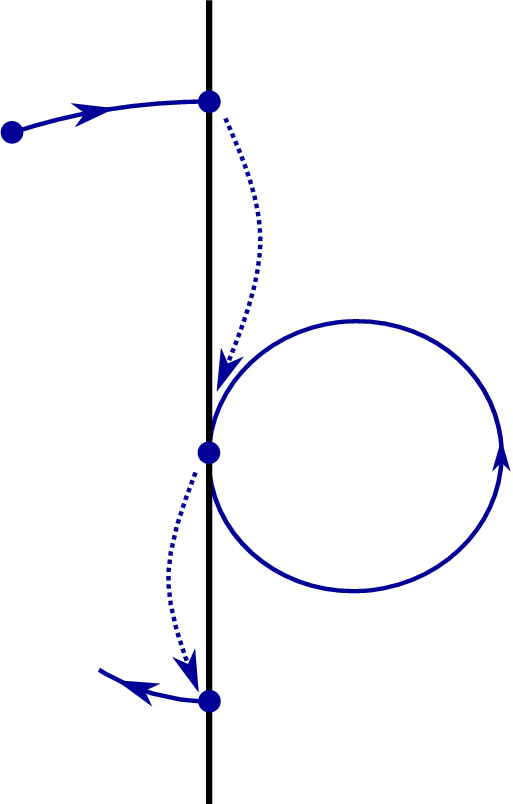} 
				%\begin{overpic}[height=5cm,grid,tics=5]{Fig10.eps} 
					\put(20,96){$\Sigma$}
					\put(0,78){$q$}
					\put(10,89){$\gamma_1$}
					\put(28,87){$q_1$}
					\put(34,70){$\varphi_n$}
					\put(28,42.5){$q_2=\overline{q_2}$}
					\put(55,25){$\gamma_2$}
					\put(28,11){$q_3$}
					\put(12,28){$\varphi_n$}
				\end{overpic}
				
				$(b)$
			\end{center}
		\end{minipage}
	\end{center}
	\caption{Illustrations of some orbits of $X$.}\label{Fig6}
\end{figure}
If $(ii)$ holds, then we cannot enter in $A_2$ by following $\gamma_2(t)$ and thus we jump to $q_3=\varphi_n(q_2)$. Let $\gamma_3(t)$ be the orbit of $X_1$ with initial condition $q_3$. Similarly to the previous case, we now try to enter in $A_1$ by following $\gamma_3$. If it is possible, then we are done. If not, then we jump again and try to enter in $A_2$ and so on. See Figure~\ref{Fig7}.
\begin{figure}[h]
	\begin{center}	
		\begin{minipage}{5cm}
			\begin{center}
				\begin{overpic}[height=5cm]{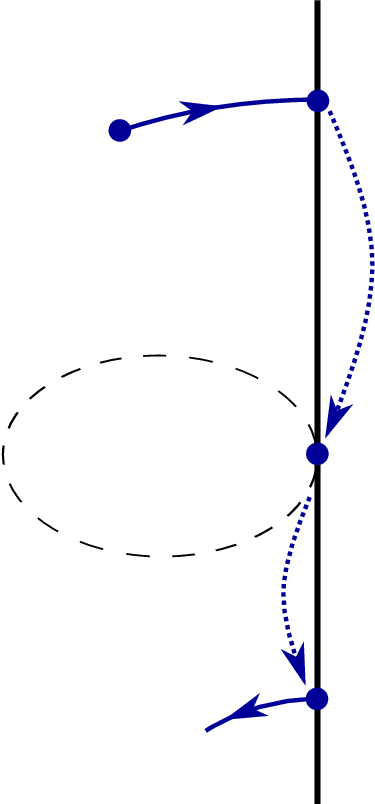} 
				%\begin{overpic}[height=5cm,grid,tics=5]{Fig11.eps} 
					\put(33,96){$\Sigma$}
					\put(13,78){$q$}
					\put(23,89.5){$\gamma_1$}
					\put(42,88){$q_1$}
					\put(47,67){$\varphi_n$}
					\put(42,42){$q_2$}
					\put(5,27){$\gamma_2$}
					\put(42,12){$q_3$}
					\put(25,5){$\gamma_3$}
					\put(26,23){$\varphi_n$}
				\end{overpic}
				
				$(a)$
			\end{center}
		\end{minipage}
		\begin{minipage}{5cm}
			\begin{center}
				\begin{overpic}[height=5cm]{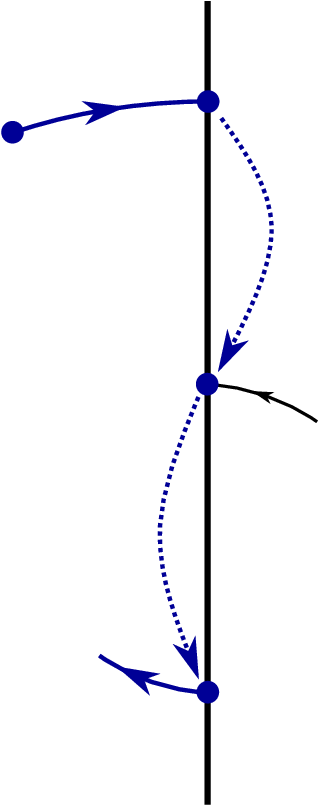} 
				%\begin{overpic}[height=5cm,grid,tics=5]{Fig13x.eps} 
					\put(19,96){$\Sigma$}
					\put(0,78){$q$}
					\put(10,89){$\gamma_1$}
					\put(28,88){$q_1$}
					\put(35,71){$\varphi_n$}
					\put(18,54){$q_2$}
					\put(35,43){$\gamma_2$}
					\put(28,12){$q_3$}
					\put(15,10){$\gamma_3$}
					\put(9,35){$\varphi_n$}
				\end{overpic}
				
				$(b)$
			\end{center}
		\end{minipage}
	\end{center}
	\caption{Illustrations of some orbits of $X$.}\label{Fig7}
\end{figure}
In particular, observe that there is the possibility of never leaving $\Sigma$. Hence, $\Gamma(q,k,t)$ is defined as the ordered concatenation of these pieces of orbits of $X_1$ and $X_2$ and the jumps between then. More precisely, let $\Gamma(q,1,t)=\gamma_1(t)$ for $0\leqslant t\leqslant t_1$, $\Gamma(q,2,t)=\gamma_2(t)$ for $0\leqslant t\leqslant t_2$ and so on (observe that we may have $t_i=0$) and define $\Gamma(q,k,t)$ as the ordered concatenation,
	\[\Gamma(q,k,t)=\Gamma(q,1,t)\cup\Gamma(q,2,t)\cup\dots\cup\Gamma(q,j,t)\cup\dots.\]
Suppose now $q\in\Sigma$. In this case $X$ have two global orbits $\Gamma_1(q,k,t)$, $\Gamma_2(q,k,t)$ with initial condition $q$. The orbit $\Gamma_1$ works as follows. Let $\gamma_1(t)$ be the orbit of $X_1$ with initial condition $q$. Suppose first that there is $t_1<0$ such that $\gamma_1(t)\in A_1$ for all $t\in[t_1,0)$ and let $q_1=\gamma_1(t_1)$. Then $q$ belongs to $\Gamma(q_1,1,t)$ and thus we just follow it. That is, we let $\Gamma_1(q,1,t)=\{q_1\}$ and $\Gamma_1(q,j,t)=\Gamma(q_1,j,t)$ for $j\geqslant2$.

Similarly, suppose that there is $t_1>0$ such that $\gamma_1(t)\in A_1$ for all $t\in(0,t_1]$ and let $q_1=\gamma_1(t_1)$. In this case we concatenate the piece of $\gamma_1$ from $q$ to $q_1$ with $\Gamma(q_1,k,t)$. That is, we let $\Gamma_1(q,1,t)=\gamma_1(t)$ for $t\in[0,t_1]$, $\Gamma_1(q,1,t)=\Gamma(q_1,1,t-t_1)$ for $t\geqslant t_1$ and $\Gamma_1(q,j,t)=\Gamma(q_1,j,t)$ for $j\geqslant2$. Similarly, one can define $\Gamma_2(q,k,t)$. 

Throughout the paper, given $q\in\mathbb{R}^2$ we may refer to the orbits $\Gamma_1(q,k,t)$ and $\Gamma_2(q,k,t)$ of $X$ with initial condition $q$. However, if $q\not\in\Sigma$, then the orbit is unique and thus $\Gamma_1=\Gamma_2$.

\section{Statement of the main results}\label{Sec2}

Let $X\in\mathfrak{X}_n$, $n\geqslant1$. As usual, to study periodic orbits and limit cycles we define a \emph{first return map} $P\colon\Sigma\to\Sigma$ (also known as the \emph{Poincar\'e map}). As we shall see in Section~\ref{Sec4}, $P$ is piecewise continuous. The reason for this is that given $y\equiv(0,y)\in\Sigma$, its first return $P(y)\in\Sigma$ can be given in four different ways, depending on the orientation of the pieces of orbits $\gamma_i$ of $X_i$, $i\in\{1,2\}$. See Figure~\ref{Fig9}.
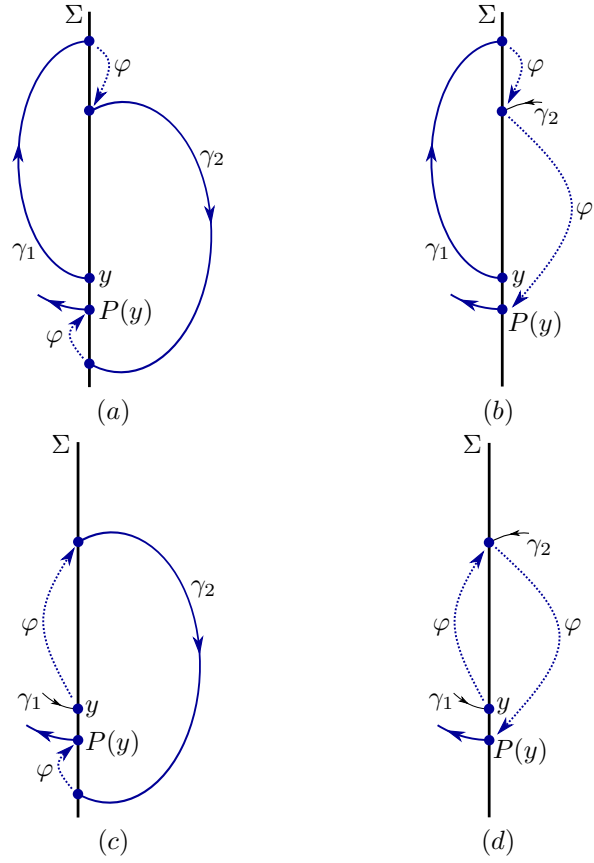
\begin{figure}[h]
	\begin{center}
		\begin{minipage}{5cm}
			\begin{center}
				\begin{overpic}[height=5cm]{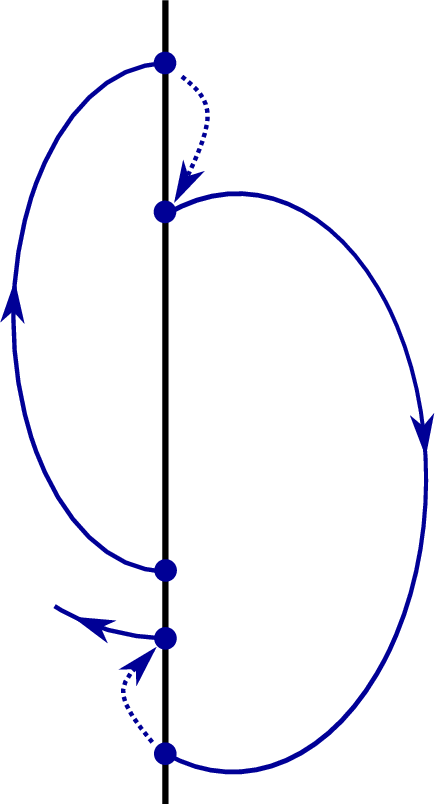} 
				%\begin{overpic}[height=5cm,grid,tics=5]{Fig1.eps} 
					\put(23,28){$y$}
					\put(23,18){$P(y)$}
					\put(27,84){$\varphi$}
					\put(9,12){$\varphi$}
					\put(14,97){$\Sigma$}
					\put(0,35){$\gamma_1$}
					\put(50,60){$\gamma_2$}
				\end{overpic}
				
				$(a)$ 
			\end{center}
		\end{minipage}
		\begin{minipage}{5cm}
			\begin{center}
				\begin{overpic}[height=5cm]{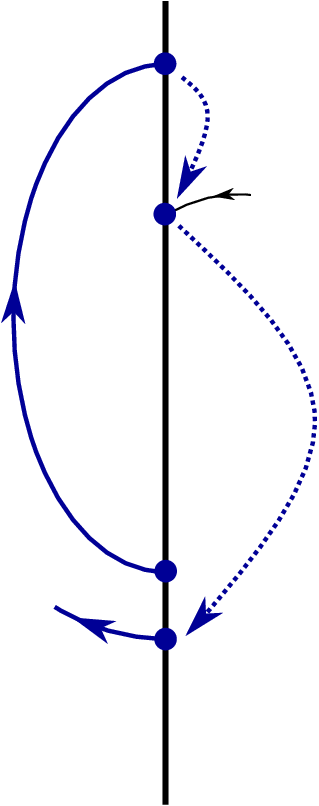} 
				%\begin{overpic}[height=5cm,grid,tics=5]{Fig18x.eps} 
					\put(23,28){$y$}
					\put(22,15){$P(y)$}
					\put(27,84){$\varphi$}
					\put(40,46){$\varphi$}
					\put(14,97){$\Sigma$}
					\put(0,35){$\gamma_1$}
					\put(29,71){$\gamma_2$}
				\end{overpic}
				
				$(b)$ 
			\end{center}
		\end{minipage}
	\end{center}
$\;$
	\begin{center}
		\begin{minipage}{5cm}
			\begin{center}
				\begin{overpic}[height=5cm]{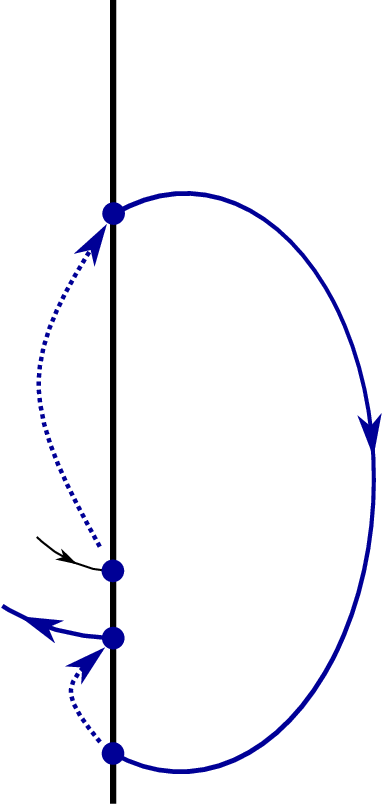} 
				%\begin{overpic}[height=5cm,grid,tics=5]{Fig19x.eps} 
					\put(16,28){$y$}
					\put(16,18){$P(y)$}
					\put(-1,50){$\varphi$}
					\put(3,11){$\varphi$}
					\put(7,97){$\Sigma$}
					\put(-2,30){$\gamma_1$}
					\put(44,60){$\gamma_2$}
				\end{overpic}
				
				$(c)$ 
			\end{center}
		\end{minipage}
		\begin{minipage}{5cm}
			\begin{center}
				\begin{overpic}[height=5cm]{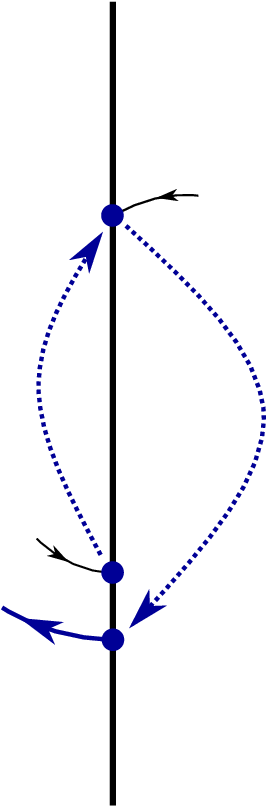} 
				%\begin{overpic}[height=5cm,grid,tics=5]{Fig20x.eps} 
					\put(16,29){$y$}
					\put(15,16){$P(y)$}
					\put(-1,50){$\varphi$}
					\put(34,50){$\varphi$}
					\put(7,97){$\Sigma$}
					\put(-2,30){$\gamma_1$}
					\put(24,71){$\gamma_2$}
				\end{overpic}
				
				$(d)$ 
			\end{center}
		\end{minipage}
	\end{center}
	\caption{Illustrations of the first return of $y$.}\label{Fig9}
\end{figure}
Let $y\in\Sigma$ and consider the global orbit $\Gamma_1(y,k,t)$. We say that $\Gamma_1$ is \emph{periodic} if $P^n(y)=y$ for some $n\geqslant1$. In particular, if the orbit through $y$ is an isolated periodic point, then we say that it is a \emph{limit cycle}. Therefore, the limit cycles of $X$ are in bijection with the isolated periodic points of $P$. If the orientation of the pieces of orbits of $X_1$ and $X_2$ agree (e.g. Figure~\ref{Fig9}$(a)$), then we say that the limit cycle is \emph{regular}. In particular, as we shall see later on, the first return map is smooth on the regular limit cycles. Moreover, if $n\in\mathbb{N}$ is the minimal $n$ such that $P^n(y)=y$, then we say that $n$ is the \emph{period} of the limit cycle. See Figure~\ref{Fig3}.
\begin{figure}[h]
	\begin{center}	
		\begin{minipage}{5cm}
			\begin{center}
				\begin{overpic}[height=5cm]{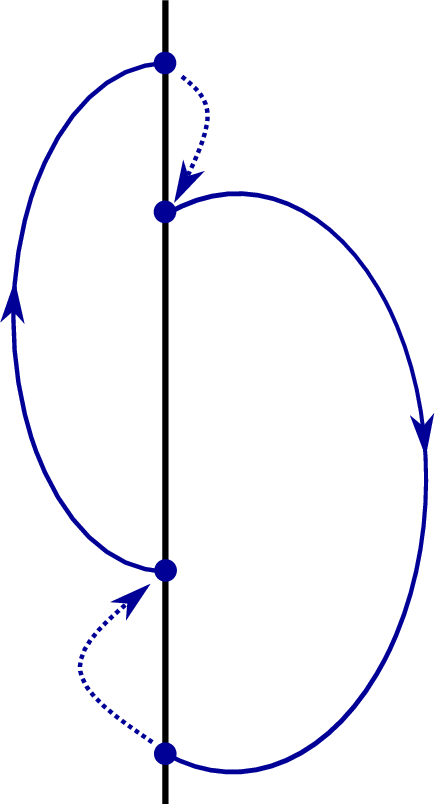} 
				%\begin{overpic}[height=5cm,grid,tics=5]{Fig22.eps}
					\put(14,97){$\Sigma$}
					\put(23,27.5){$y=P(y)$} 
				\end{overpic}
				
				Period one.
			\end{center}
		\end{minipage}
		\begin{minipage}{5cm}
			\begin{center}
				\begin{overpic}[height=5cm]{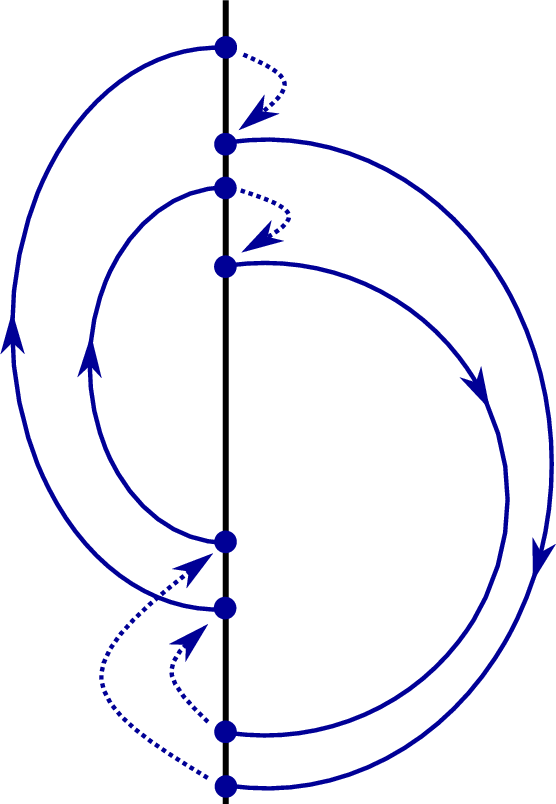} 
				%\begin{overpic}[height=5cm,grid,tics=5]{Fig2.eps} 
					\put(22,97){$\Sigma$}
					\put(31,31){$y=P^2(y)$}
					\put(30,22){$P(y)$} 
				\end{overpic}
				
				Period two.
			\end{center}
		\end{minipage}
	\end{center}
	\caption{Illustrations of regular limit cycles.}\label{Fig3}
\end{figure}
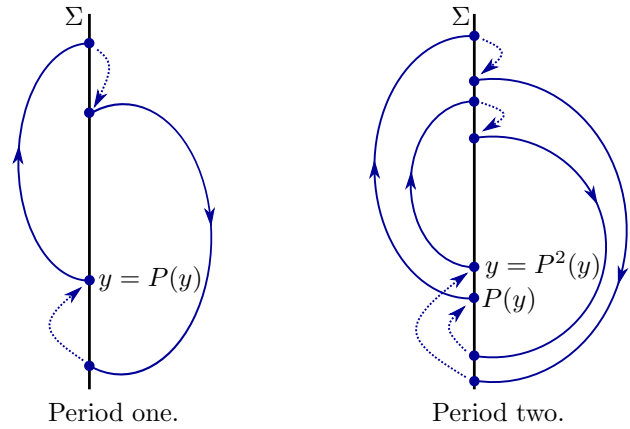
We can also define the \emph{displacement map} $d\colon\Sigma\to\Sigma$ given by $d(y)=P(y)-y$, which is also smooth on the regular limit cycles. Hence, let $\Gamma$ be a regular limit cycle of $X$ and let $y_0\in\Sigma$ be its associated zero of the displacement map. We say that $\Gamma$ is \emph{hyperbolic} if $d'(y_0)\neq0$. Moreover, in this case we say that $\Gamma$ is \emph{stable} (resp. \emph{unstable}) if $d'(y_0)<0$ (resp. $d'(y_0)>0$). In our first main result we study the regular limit cycles of $\mathfrak{X}_1$.

\begin{theorem}\label{Main1}
	Let $X=(X_1,X_2;\Sigma;\varphi)\in\mathfrak{X}_1$. If $X$ has a regular limit cycle, then $|\varphi'(0)|\neq1$. Moreover, if it exists then it is unique, hyperbolic, stable (resp. unstable) if $|\varphi'(0)|<1$ (resp. $|\varphi'(0)|>1$) and has period one.
\end{theorem}

We observe that the necessary condition $|\varphi'(0)|\neq1$ for regular limit cycles agree with the fact that if $\varphi=\text{Id}_{\Sigma}$ (i.e. if $\varphi$ is the identity map $\text{Id}\colon\Sigma\to\Sigma$), then the hybrid system becomes a \emph{Filippov system} and thus it follows from Llibre and Teixeira \cite{LliTei2018} that it cannot have limit cycles.

We also observe that the authors in \cite{LiLiu2023} also worked with planar hybrid systems composed by two linear centers. More precisely, let $\mathfrak{X}^*=(X_1,X_2;\Sigma;\psi)$ be the family of planar hybrid systems such that $X_1$ and $X_2$ are linear vector fields with center singularities, $\Sigma$ is the straight line given by $x=0$ and $\psi\colon\Sigma\to\Sigma$ is given by
\[\psi(y)=\left\{\begin{array}{ll} r_+y, & \text{if } y\geqslant0, \vspace{0.2cm} \\ r_-y, & \text{if } y\leqslant0, \end{array}\right.\]
with $r_\pm\in(0,1]$. The authors in \cite{LiLiu2023} proved that if $X\in\mathfrak{X}^*$, then $X$ can have at most one limit cycle and that this upper bound is sharp (i.e. there is $X\in\mathfrak{X}^*$ with one limit cycle). Moreover, we observe that in their paper the hypothesis $r_\pm\in(0,1]$ ensures that any limit cycle is regular. 

In our second main result we prove that for almost all $X\in\mathfrak{X}_n$, $n\geqslant1$, there is a compact $K\subset\mathbb{R}^2$ such that any relevant dynamics of $X$ happens inside $K$. 

\begin{theorem}\label{Main2}
	Let $X=(X_1,X_2;\Sigma;\varphi)\in\mathfrak{X}_n$ and if $n=1$, suppose $|\varphi'(0)|\neq1$. Then there is a compact $K\subset\mathbb{R}^2$ such that if $q\in\mathbb{R}^2\backslash K$, then the following statements hold.
	\begin{enumerate}[label=(\alph*)]
		\item If $n\geqslant2$, then  $|\Gamma_i(q,j,t)|\to\infty$ as $j\to\infty$, for $i\in\{1,2\}$.
		\item If $n=1$ and $|\varphi'(0)|>1$, then  $|\Gamma_i(q,j,t)|\to\infty$ as $j\to\infty$, for $i\in\{1,2\}$.
		\item If $n=1$ and $|\varphi'(0)|<1$, then $\Gamma_i(q,j,t)\to K$ as $j\to\infty$, for $i\in\{1,2\}$.
	\end{enumerate}
\end{theorem}

Theorem~\ref{Main2} shows that outside the relevant compact $K$, every orbit either escape to the infinity (statements $(a)$ and $(b)$), or converges to $K$ (statement $(c)$). However, as we shall prove in our next main result, the dynamics inside such compact may be very rich. For an illustration of $K$, see Figure~\ref{Fig4}.

Given a set $K\subset\mathbb{R}^2$, we say that $K\subset\mathbb{R}^2$ is \emph{positive invariant} by $X$ if for every $q\in K$ we have $\Gamma_i(q,j,t)\in K$ for every $j\in\mathbb{N}$, $t\geqslant0$ and $i\in\{1,2\}$. Moreover, given $q\in\mathbb{R}^2$ and $\varepsilon>0$, let $B_\varepsilon(q)$ denote the open ball of radius $\varepsilon$ centered at $q$. 

\begin{definition}\label{Def1}
	Let $X\in\mathfrak{X}_n$, $n\geqslant 1$. We say that $X$ is chaotic if there is a positive invariant compact set $K\subset\mathbb{R}^2$ satisfying the following statements.
	\begin{enumerate}[label=(\roman*)]
		\item The periodic orbits of $X$ are dense in $K$.
		\item Given any two nonempty open sets $U_1$, $U_2\subset K$, there are $q\in U_1$, $j_0\in\mathbb{N}$ and $t_0>0$ such that $\Gamma_i(q,j_0,t_0)\in U_2$, for some $i\in\{1,2\}$.
		\item There is $\beta>0$ such that for every $q_1\in K$ and $\varepsilon>0$ there are $q_2\in B_\varepsilon(q_1)\cap K$, $j_0\in\mathbb{N}$, $t_0>0$ and $i_0\in\{1,2\}$ such that
			\[||\Gamma_{i_0}(q_1,j_0,t_0)-\Gamma_{i_0}(q_2,j_0,t_0)||>\beta,\]
		where $||\cdot||$ denotes the standard norm of $\mathbb{R}^2$.
	\end{enumerate}
\end{definition}

As far as we know, the definition of chaos in Filippov system was first introduced by Colombo and Jeffrey~\cite{ColJef}, to study the local dynamics of two-folds singularities in $\mathbb{R}^3$. Furthermore, chaos in \emph{planar} Filippov system was first introduced by Buzzi et al \cite{BuzCarEuz2016}. Both definitions resembles the notion of \emph{nondeterministic chaos}, which in turn resembles the famous definition of Devaney~\cite[Chapter $8$]{Devaney}, but without the density of periodic orbits (i.e. without statement $(i)$ in Definition~\ref{Def1}).

We have chose to bring back the notion of density of periodic orbits because the previous definitions were focused on the loss of determinism in the neighborhood of a single point (see \cite[p. $436$]{ColJef}), while we are focused on the loss of determinism in a non-local framework. Hence, we think that in our framework the density of periodic orbits is an enrichment for the notion of chaos and thus our definition resembles more to Devaney's one.

As anticipated by the above discussion, in our final main result we provide an example of a chaotic system.

\begin{theorem}\label{Main3}
	There is $X\in\mathfrak{X}_2$ chaotic.
\end{theorem}

As the reader shall see, we observe that our example of chaos does not rely on \emph{sliding dynamics} through the switching manifold.

For more details about chaos in Filippov systems, we refer to the works of Carvalho and coauthors \cites{BuzCarEuz2016,Car2024,Car2024x2,CarEuz2020}. For a example of a chaotic hybrid dynamical system of dimension one, we refer to Corron et al~\cite{Corron}.

The paper is organized as follows. In Section~\ref{Sec3} we have some preliminaries results about normal forms and chaos. The main theorems are proved in Sections~\ref{Sec4}, \ref{Sec5} and \ref{Sec6}. In Section~\ref{Sec7} we dive in more details about this piecewise structure of the first return map. Finally, at Section~\ref{Sec8} we have a conclusion and some further thoughts.

\section{Preliminaries}\label{Sec3}

\subsection{Normal form}\label{Sec3.1}

In this section we recall the normal form of a linear center and its first integral, due to Llibre and Teixeira~\cite{LliTei2018}. More precisely, the normal form follows from \cite[Lemma~$1$]{LliTei2018}, while the first integral follows from simple calculations.

\begin{lemma}\label{Lemma1}
	Let $X=(P,Q)$ be a planar linear vector field with a center singularity. Then $X$ can be written as
	\begin{equation}\label{1}
		P(x,y)=-bx-\delta(4b^2+\omega^2)y+d, \quad Q(x,y)=\delta x+by+c,
	\end{equation}
	with $\omega\neq0$ and $\delta\in\{-1,1\}$. Moreover, \eqref{1} has a quadratic first integral given by,
	\begin{equation}\label{2}
		H(x,y)=2cx-2dy+\delta x^2+2bxy+\delta(4b^2+\omega^2)y^2.
	\end{equation}
\end{lemma}

We observe that in its original state Lemma~\ref{Lemma1} in Llibre and Teixeira~\cite{LliTei2018}, instead of $\delta\in\{-1,1\}$, they had an extra coefficient $a$. But resclaing the time by $a$ or $-a$ we obtain Lemma~\ref{Lemma1}. Given $X=(X_1,X_2;\Sigma;\varphi_n)\in\mathfrak{X}_n$, $n\geqslant1$, we can suppose that $X_1$ and $X_2$ are written as in \eqref{1}, simultaneously. That is, we can suppose that $X_i=(P_i,Q_i)$ is given by
\begin{equation}\label{19}
	P_i(x,y)=-b_ix-\delta_i(4b_i^2+\omega_i^2)y+d_i, \quad Q_i(x,y)=\delta_i x+b_iy+c_i,
\end{equation}
with $\omega_i\neq0$ and $\delta_i\in\{-1,1\}$, $i\in\{1,2\}$. Hence, it follows from \eqref{2} that $X_i$ has a first integral given by
\begin{equation}\label{20}
	H_i(x,y)=2c_ix-2d_iy+\delta x^2+2b_ixy+\delta_i(4b_i^2+\omega_i^2)y^2,
\end{equation}
for $i\in\{1,2\}$.

\subsection{Chaos}\label{Sec3.2}

Let $f\colon I\to I$ be a map with $I\subset\mathbb{R}$ a compact interval (i.e. a closed interval of finite length). Following Devaney \cite[Chapter $8$]{Devaney} we say that $f$ is \emph{chaotic} if the following statements hold.
\begin{enumerate}[label=(\alph*)]
	\item The set of periodic points of $f$ is dense in $I$;
	\item $f$ is transitive on $I$; that is, given any two non-empty subintervals $U_1$ and $U_2$ of $I$, there is $x\in U_1$ and $n>0$ such that $f^n(x)\in U_2$.
\end{enumerate}

\begin{remark}\label{Remark1}
	The original definition of chaos, due to Devaney in the first edition of his book \cite{Devaney}, also included the following condition.	
	\begin{enumerate}[label=(\alph*)]
		\addtocounter{enumi}{2}
		\item $f$ has sensitive dependence on $I$; that is, there is a \emph{sensitivity constant} $\beta>0$ such that for any $x_0\in I$ and any neighborhood $U\subset I$ of $x_0$, there is $y_0\in U$ and $n>0$ such that $|f^n(x_0)-f^n(y_0)|>\beta$.
	\end{enumerate} 
	However, Banks et al \cite{Banks} proved that this third condition is redundant. If $(a)$ and $(b)$ hold, then $(c)$ holds. Moreover, consider the following condition.
	\begin{enumerate}
		\item[$(b')$] $f$ has a dense orbit. That is, there is $x\in I$ such that $\{f^n(x)\colon n\geqslant0\}$ is dense in $I$.
	\end{enumerate} 
	We observe that if $(b')$ holds, then $(b)$ holds. In fact, they are equivalent. See \cite[Lemma $4.3.4$]{ViaOli2016}.
\end{remark}

A common example of chaotic map is the \emph{logistic map} $f_\lambda\colon[0,1]\to[0,1]$ given by $f_\lambda(x)=\lambda x(1-x)$, with $\lambda=4$.

\begin{proposition}\label{P1}
	The logistic map $f_4\colon[0,1]\to[0,1]$ is chaotic.
\end{proposition}

For a proof of Proposition~\ref{P1}, see \cite[p. $63$]{Devaney}. Consider two maps $f\colon I\to I$ and $g\colon J\to J$, with $I$ and $J$ compact intervals of $\mathbb{R}$. We say that $f$ and $g$ are \emph{conjugated} if there is a homeomorphism $h\colon I\to J$ such that $h\circ f=g\circ h$.

\begin{proposition}\label{P2}
	Let $f\colon I\to I$ and $g\colon J\to J$ two conjugated maps, with $I$ and $J$ compact intervals. If $f$ is chaotic, then $g$ is chaotic.
\end{proposition}

For a proof of Proposition~\ref{P2}, we refer to Hirsch et al \cite[p. $340$]{HirSmaDev2004}. Another well-known example of chaotic map is the \emph{tent map} $T\colon[0,1]\to[0,1]$ given by
	\[T(x)=\left\{\begin{array}{ll} 2x, & \text{if } 0\leqslant x\leqslant\frac{1}{2}, \vspace{0.2cm} \\ 2-2x, & \text{if } \frac{1}{2}\leqslant x\leqslant 1. \end{array}\right.\]

\begin{proposition}\label{P3}
	The tent map $T$ and the logistic map $f_4$ are conjugated. In particular, $T$ is chaotic.
\end{proposition}

\begin{proof} Let $h\colon[0,1]\to[0,1]$ be given by
	\[h(x)=\sin^2\left(\frac{\pi}{2}x\right),\]
and observe that $h$ is a homeomorphism. We claim that $h\circ T=f\circ h$, i.e. $T$ and $f_4$ are conjugated. Indeed, if $0\leqslant x\leqslant\frac{1}{2}$ then,
	\[h(T(x))=h(2x)=\sin^2(\pi x)=4\sin^2\left(\frac{\pi}{2}x\right)\cos^2\left(\frac{\pi}{2}x\right)=4h(x)(1-h(x))=f_4(h(x)).\]
Similarly, if $\frac{1}{2}\leqslant x\leqslant1$ then,
	\[h(T(x))=h(2(1-x))=\sin^2(\pi(1-x))=\sin^2(\pi x)=f_4(h(x)).\]
Therefore $f_4$ and $T$ are conjugated and thus it follows from Propositions~\ref{P1} and \ref{P2} that $T$ is chaotic. \end{proof}

Given $x\in[0,1]$, let $[x]_2=0.s_1s_2s_3\dots$ be an expression of $x$ in binary (in particular, observe that if $x=1$, then $s_i=1$ for every $i\geqslant1$). Note that the expression of $T(x)$ in binary is given by
	\[[T(x)]_2=\left\{\begin{array}{ll} 0.s_2s_3\dots, & \text{if } s_1=0, \vspace{0.2cm} \\ 0.s_2^*s_3^*\dots, & \text{if } s_1=1, \end{array}\right.\]
where $s_i^*=1-s_i$. 

\begin{proposition}\label{P4}
	The map $T^2\colon[0,1]\to[0,1]$ is chaotic.
\end{proposition}

\begin{proof} Let $F=T^2$. We now prove that $F$ satisfies statements $(a)$ and $(b)$ of the definition of chaos. To prove statement $(a)$, observe that if $T^k(x)=x$ for some $k\in\mathbb{N}$, then
		\[F^k(x)=T^{2k}(x)=T^k(T^k(x))=T^k(x)=x.\]
Hence, if $x$ is a periodic point of $T$, then it is also a periodic point of $F$. Since the periodic points of $T$ are dense in $[0,1]$ (because $T$ is chaotic), it follows that the periodic points of $F$ are dense in $[0,1]$ and thus we have statement $(a)$. We now prove statement $(b)$. First, observe that if $[x]_2=0.s_1s_2s_3\dots$, then
		\[[F(x)]_2=\left\{\begin{array}{ll} 0.s_3\dots, & \text{if } s_2=0, \vspace{0.2cm} \\ 0.s_3^*\dots, & \text{if } s_2=1. \end{array}\right.\]
In particular, the expression of $F(x)$ in binary does not depend on $s_1$. Similarly, it follows by induction that
\begin{equation}\label{17}
	[F^n(x)]_2=\left\{\begin{array}{ll} 0.s_{2n+1}\dots, & \text{if } s_{2n}=0, \vspace{0.2cm} \\ 0.s_{2n+1}^*\dots, & \text{if } s_{2n}=1, \end{array}\right.
\end{equation}
for every $n\geqslant1$. In particular, the expression of $F^n(x)$ in binary depend only on $s_{2n}$. We claim that $F$ has a dense orbit. Indeed, let $x^*\in[0,1]$ be given by,
	\[[x^*]_2=0.\underbrace{00\;10}_{2-\text{blocks}}|\underbrace{0000\;0010\;0100\;1000\;0110\;1010\;1100\;1110\;}_{4-\text{blocks}}|\underbrace{\dots}_{6-\text{blocks}}|\dots.\]
That is, $x^*$ is constructed by listing out all possible sequence of blocks of $0$'s and $1$'s of even length and such that the last digit of each block is $0$. Since each block end with $0$, it follows from \eqref{17} that given a block of length $2n$ we can shift it by applying $F^{2n}$. Hence, given $x\in[0,1]$ and $N\geqslant1$, it follows that there is an iteration of $x^*$ such that the first $2N+1$ digits of $[x]_2$ and $[x^*]_2$ are equal. Therefore the orbit of $x^*$ is dense in $[0,1]$ and thus it follows from Remark~\ref{Remark1} that $F$ is transitive, proving statement $(b)$. \end{proof}

Let $f\colon I\to I$ and $g\colon J\to J$ be conjugated by $h$, i.e. $h\circ f=g\circ h$. It follows by induction that $h\circ f^n=g^n\circ h$ for every $n\geqslant 1$. That is, if $f$ and $g$ are conjugated, then $f^n$ and $g^n$ are also conjugated. Therefore, it follows from Proposition~\ref{P3} that $f_4^2$ is conjugated to $T^2$ and thus it follows from Proposition~\ref{P4} that $f_4^2$ is chaotic. Hence, we have the following result.

\begin{proposition}\label{P5}
	Let $f_4\colon[0,1]\to[0,1]$ be the logistic map $f_4(x)=4x(1-x)$. Then $f_4^2$ is chaotic.
\end{proposition}

\section{Proof of Theorem~\ref{Main1}}\label{Sec4}

Let $X=(X_1,X_2;\Sigma;\varphi)\in\mathfrak{X}_n$. We remind from Section~\ref{Sec3.1} that we can suppose that $X_i=(P_i,Q_i)$ is given by
\begin{equation}\label{13}
	P_i(x,y)=-b_ix-\delta_i(4b_i^2+\omega_i^2)y+d_i, \quad Q_i(x,y)=\delta_i x+b_iy+c_i,
\end{equation}
with $\omega_i\neq0$ and $\delta_i\in\{-1,1\}$, $i\in\{1,2\}$. Moreover, it also follows from Section~\ref{Sec3.1} that $X_i$ has a first integral given by
\begin{equation}\label{3}
	H_i(x,y)=2c_ix-2d_iy+\delta x^2+2b_ixy+\delta_i(4b_i^2+\omega_i^2)y^2,
\end{equation}
for $i\in\{1,2\}$. We now work with the first return map $P\colon\Sigma\to\Sigma$. Suppose first $X\in\mathfrak{X}_1$ and let $\varphi(y)=ay+b$. Let $y_1\in\Sigma$. Since the orbits of $X_1$ are ellipses, it follows that the orbit of $X_1$ through $y_1$ intersect $\Sigma$ in a point $y_2$ that satisfies, 
\begin{equation}\label{8}
	H_1(0,y_2)-H_1(0,y_1)=0.
\end{equation}
It follows from \eqref{3} that \eqref{8} is given by, 
\begin{equation}\label{9}
	(y_1-y_2)\bigl(2d_1-(4b_1^2+\omega_1^2)\delta_1(y_1+y_2)\bigr)=0.
\end{equation}
Since we are interested in the solution other than $y_2=y_1$, we conclude from \eqref{9} that,
\begin{equation}\label{4}
	y_2=-y_1+\eta_1, \quad \eta_1= \frac{2d_1\delta_1}{4b_1^2+\omega_1^2}.
\end{equation}
Observe that it is still possible to obtain $y_2=y_1$ from \eqref{4} in the particular case of $y_1=\eta_1/2$. This is the case where there is a unique ellipse of $X_1$ tangent to $\Sigma$, or the case in which $y_1$ is the center of $X_1$. Let $\gamma_1$ be the orbit of $X_1$ that contains $y_1$ and $y_2$. Suppose first that the direction of $\gamma_1$ is from $y_1$ to $y_2$, see Figure~\ref{Fig1}(a).
\begin{figure}[h]
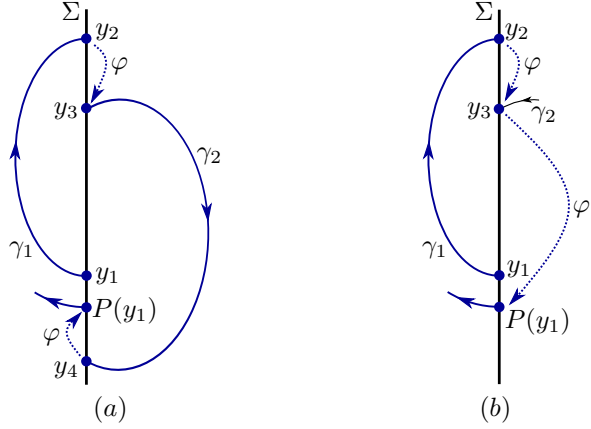

	\begin{center}
		\begin{minipage}{5cm}
			\begin{center}
				\begin{overpic}[height=5cm]{Fig1.eps} 
				%\begin{overpic}[height=5cm,grid,tics=5]{Fig1.eps} 
					\put(23,28){$y_1$}
					\put(23,93){$y_2$}
					\put(12,72){$y_3$}
					\put(12,3){$y_4$}
					\put(22,18){$P(y_1)$}
					\put(27,84){$\varphi$}
					\put(9,12){$\varphi$}
					\put(14,97){$\Sigma$}
					\put(0,35){$\gamma_1$}
					\put(50,60){$\gamma_2$}
				\end{overpic}
				
				$(a)$ 
			\end{center}
		\end{minipage}
		\begin{minipage}{5cm}
			\begin{center}
				\begin{overpic}[height=5cm]{Fig18x.eps} 
				%\begin{overpic}[height=5cm,grid,tics=5]{Fig18x.eps} 
					\put(22.5,30){$y_1$}
					\put(22,93){$y_2$}
					\put(12,72){$y_3$}
					\put(22,15){$P(y_1)$}
					\put(27,84){$\varphi$}
					\put(40,46){$\varphi$}
					\put(14,97){$\Sigma$}
					\put(0,35){$\gamma_1$}
					\put(29,71){$\gamma_2$}
				\end{overpic}
				
				$(b)$ 
			\end{center}
		\end{minipage}
	\end{center}
\caption{Illustrations of the first return of $y_1$.}\label{Fig1}
\end{figure}
Then we follow $\gamma_1$ until we reach $y_2$, and after we jump to $y_3=\varphi(y_2)$. Since $\varphi(y)=ay+b$, it follows that
\begin{equation}\label{5}
	y_3=ay_2+b.
\end{equation}
Similarly let $\gamma_2$ be the orbit of $X_2$ through $y_3$ and let also $y_4\in\Sigma$ be the other point of intersection of $\gamma_2$ with $\Sigma$. Similarly to \eqref{4} we leave placed
\begin{equation}\label{6}
	y_4=-y_3+\eta_2, \quad \eta_2=\frac{2d_2\delta_2}{4b_2^2+\omega_2^2}.
\end{equation}
If the direction of $\gamma_2$ is from $y_3$ to $y_4$, then we follow $\gamma_2$ until we reach $y_4$ and then we jump to $P(y_1)=\varphi(y_4)$, i.e.
\begin{equation}\label{7}
	P(y_1)=ay_4+b.
\end{equation}
Therefore it follows from \eqref{4}, \eqref{5}, \eqref{6} and \eqref{7} that,
\begin{equation}\label{12x1}
	P(y_1)=a^2y_1+\beta_1, \quad \beta_1=b(1-a)+2a(\eta_2-a\eta_1).
\end{equation}
Now if the direction of $\gamma_2$ is from $y_4$ to $y_3$, then we cannot enter in $A_2$ and thus from $y_3$ we jump to $P(y_1)=\varphi(y_3)$. See Figure~\ref{Fig1}(b). In this case we have
\begin{equation}\label{12x2}
	P(y_1)=-a^2y_1+\beta_2, \quad \beta_2=b(1+a)+2a^2\eta_1.
\end{equation}
Similarly we have the two cases in which the direction of $\gamma_1$ does not agree. In these cases, first we jump to $y_3=\varphi(y_1)$. If we can enter in $A_2$ by following $\gamma_2$, then we follow it until we reach $y_4$ and then we jump to $P(y_1)=\varphi(y_4)$. See Figure~\ref{Fig10}(a). If not, then we jump to $P(y_1)=\varphi(y_3)$. See Figure~\ref{Fig10}(b).
\begin{figure}[h]
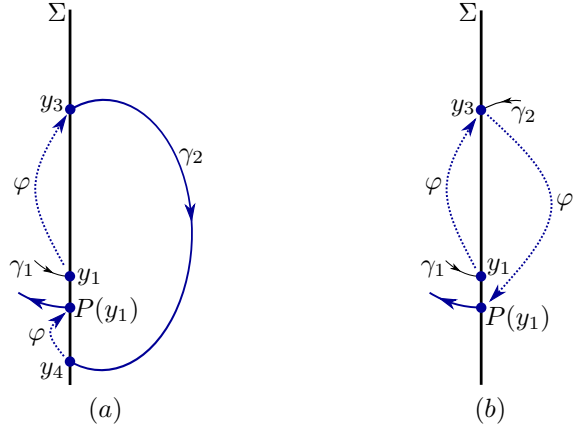

	\begin{center}
		\begin{minipage}{5cm}
			\begin{center}
				\begin{overpic}[height=5cm]{Fig19x.eps} 
				%\begin{overpic}[height=5cm,grid,tics=5]{Fig19x.eps} 
					\put(16,28){$y_1$}
					\put(6,74){$y_3$}
					\put(6,3){$y_4$}
					\put(15,18){$P(y_1)$}
					\put(-1,52){$\varphi$}
					\put(3,12){$\varphi$}
					\put(8,97){$\Sigma$}
					\put(-2,30){$\gamma_1$}
					\put(43,60){$\gamma_2$}
				\end{overpic}
				
				$(a)$ 
			\end{center}
		\end{minipage}
		\begin{minipage}{5cm}
			\begin{center}
				\begin{overpic}[height=5cm]{Fig20x.eps} 
				%\begin{overpic}[height=5cm,grid,tics=5]{Fig20x.eps} 
					\put(15.5,30.5){$y_1$}
					\put(6,73){$y_3$}
					\put(15,16){$P(y_1)$}
					\put(-1,52){$\varphi$}
					\put(34,48){$\varphi$}
					\put(8,97){$\Sigma$}
					\put(-2,30){$\gamma_1$}
					\put(22,71){$\gamma_2$}
				\end{overpic}
				
				$(b)$ 
			\end{center}
		\end{minipage}
	\end{center}
	\caption{Illustrations of the first return of $y_1$.}\label{Fig10}
\end{figure}
In the former we have
\begin{equation}\label{12x3}
	P(y_1)=-a^2y_1+\beta_3, \quad \beta_3=b(1-a)+2a\eta_2.
\end{equation}
While in the later we have,
\begin{equation}\label{12x4}
	P(y_1)=a^2y_1+\beta_4, \quad \beta_4=b(1+a).
\end{equation}
Hence there are $J_j\subset\mathbb{R}$ with $j\in\{1,2,3,4\}$ and $\cup_jJ_j=\mathbb{R}$, such that the first return map $P\colon\Sigma\to\Sigma$ is given by,
\begin{equation}\label{10}
	P(y)=\left\{\begin{array}{rl}
					a^2y+\beta_1 & \text{if } y\in J_1, \vspace{0.2cm} \\
					-a^2y+\beta_2 & \text{if } y\in J_2, \vspace{0.2cm} \\
					a^2y+\beta_3 & \text{if } y\in J_3, \vspace{0.2cm} \\
					-a^2y+\beta_4 & \text{if } y\in J_4. \vspace{0.2cm} \\
				\end{array}\right.
\end{equation}
We provide more details about the sets $J_j$ in Section~\ref{Sec7}. From now on in order to seek for regular limit cycles, we assume that the orientation of the orbits of $X_1$ and $X_2$ agree and thus $P(y)=a^2y+\beta_1$. It follows by induction that,
	\[P^n(y)=a^{2n}y+\beta_1\sum_{k=0}^{n-1}a^{2k}.\]
Hence we have $P^n(y)=y$ if, and only if,
	\[(a^{2n}-1)y+\beta_1\sum_{k=0}^{n-1}a^{2k}=0.\]
Observe that if $|a|=1$, then either we have no regular periodic orbits (if $\beta_1\neq0$), or every $y\in J_1$ is a regular periodic orbit (if $\beta_1=0$). Therefore if we have a regular limit cycle, then $|a|\neq1$ and
\begin{equation}\label{15}
	y=\frac{\beta_1}{1-a^{2n}}\sum_{k=0}^{n-1}a^{2k}=\frac{\beta_1}{1-a^2},
\end{equation}
where the last equality follows from the well-known property of the geometric series given by,
	\[\sum_{k=0}^{n-1}a^{2k}=\frac{1-a^{2n}}{1-a^2}.\]
However observe that the outermost right-hand side of \eqref{15} is also the unique solution of $P(y)=y$. Therefore if we have a regular limit cycle $\Gamma$, then it is the unique regular limit cycle with period one. Moreover, it follows from continuity that there is a neighborhood of $\Gamma$ such that the orientation of the piece of orbits of $X_1$ and $X_2$ agree. Hence in this neighborhood the displacement map is given by
	\[d(y)=(a^2-1)y+\beta_1,\]
and thus $\Gamma$ is hyperbolic and stable (resp. unstable) if $|a|<1$ (resp. $|a|>1$).

\section{Proof of Theorem~\ref{Main2}}\label{Sec5}

Suppose first $n=1$, $\varphi(y)=ay+b$ and $|a|\neq1$. Let $\beta=\max\{|\beta_i|\colon i\in\{1,2,3,4\}\}$, with $\beta_i$ given by \eqref{12x1}, \eqref{12x2}, \eqref{12x3} and \eqref{12x4}. It follows from \eqref{10} that
\begin{equation}\label{11}
	P^n(y)=\delta a^{2n}y+\sum_{k=0}^{n-1}a^{2k}\xi_k,
\end{equation}
with $\delta\in\{-1,1\}$ and $|\xi_k|\in\{\beta_i\colon i\in\{1,2,3,4\}\}$. We recall that given $u$, $v\in\mathbb{R}$, it is well known that,
	\[|u|-|v|\leqslant|u+v|\leqslant|u|+|v|.\]
Hence, it follows from \eqref{11} that,
\begin{equation}\label{12}
	 a^{2n}|y|-\beta\sum_{k=0}^{n-1}a^{2k}\leqslant|P^n(y)|\leqslant a^{2n}|y|+\beta\sum_{k=0}^{n-1}a^{2k}.
\end{equation}
From the outermost right-hand side of \eqref{12} we have that if $|a|<1$, then
	\[\limsup\limits_{n}|P^n(y)|\leqslant \lim\limits_{n}\left(a^{2n}|y|+\beta\sum_{k=0}^{n-1}a^{2k}\right)=\beta\lim\limits_{n}\frac{1-a^{2n}}{1-a^2}=\frac{\beta}{1-a^2}.\]
Hence, there is a compact $K\subset\mathbb{R}^2$ such that all orbits of $X$ outside $K$ converges to it. See Figure~\ref{Fig4}.
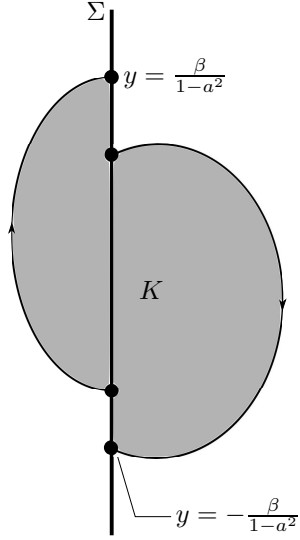
\begin{figure}[h]
	\begin{center}
		\begin{overpic}[height=7cm]{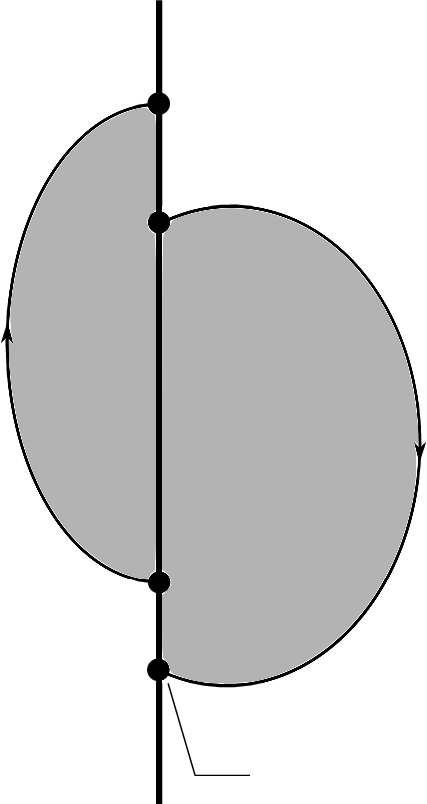} 
		%\begin{overpic}[height=7cm,grid,tics=5]{Fig23.eps} 	
			\put(22,86){$y=\frac{\beta}{1-a^2}$}
			\put(32,2.5){$y=-\frac{\beta}{1-a^2}$}
			\put(25,45){$K$}
			\put(15,98){$\Sigma$}
		\end{overpic}		
	\end{center}
	\caption{Illustration of the compact $K$.}\label{Fig4}
\end{figure}
Similarly, it follows from the outermost left-hand side of \eqref{12} that if $|a|>1$, then
	\[|P^n(y)|\geqslant a^{2n}|y|-\beta\frac{a^{2n}-1}{a^2-1}=a^{2n}\left(|y|-\frac{\beta}{a^2-1}\right)+\frac{\beta}{a^2-1}.\]
Therefore if $|y|>\beta/(a^2-1)$, then $|P^n(y)|\to\infty$ as $n\to\infty$. That is, if $|a|>1$, then there is a compact $K\subset\mathbb{R}^2$ such that every orbit of $X$ outside $K$ goes to the infinity. 

We now work on the general case $X\in\mathfrak{X}_n$, $n\geqslant2$. It follows similarly from the proof of Theorem~\ref{Main1} that the first return map $P\colon\Sigma\to\Sigma$ is given by
\begin{equation}\label{14}
	P(y)=\left\{\begin{array}{ll}
		\varphi_n(\eta_2-\varphi_n(\eta_1-y)) & \text{if } y\in J_1, \vspace{0.2cm} \\
		\varphi_n(\varphi_n(\eta_1-y)) & \text{if } y\in J_2, \vspace{0.2cm} \\
		\varphi_n(\eta_2-\varphi_n(y)) & \text{if } y\in J_3, \vspace{0.2cm} \\
		\varphi_n(\varphi_n(y)) & \text{if } y\in J_4, \vspace{0.2cm} \\
	\end{array}\right.
\end{equation}	
with $\cup_j J_j=\mathbb{R}$ and $\eta_i=2d_i\delta_i/(4b_i^2+\omega_i^2)$, $i\in\{1,2\}$. In particular it follows from \eqref{14} that $P$ is piecewise polynomial. Since $\varphi_n$ is nonlinear, it follows that each branch of \eqref{14} is also nonlinear. Therefore there is $Y_0>0$ such that if $|y|>Y_0$, then $|P(y)|>|y|$, see Figure~\ref{Fig2}.
\begin{figure}[h]
	\begin{center}
		\begin{minipage}{5cm}
			\begin{center}
				\begin{overpic}[height=5cm]{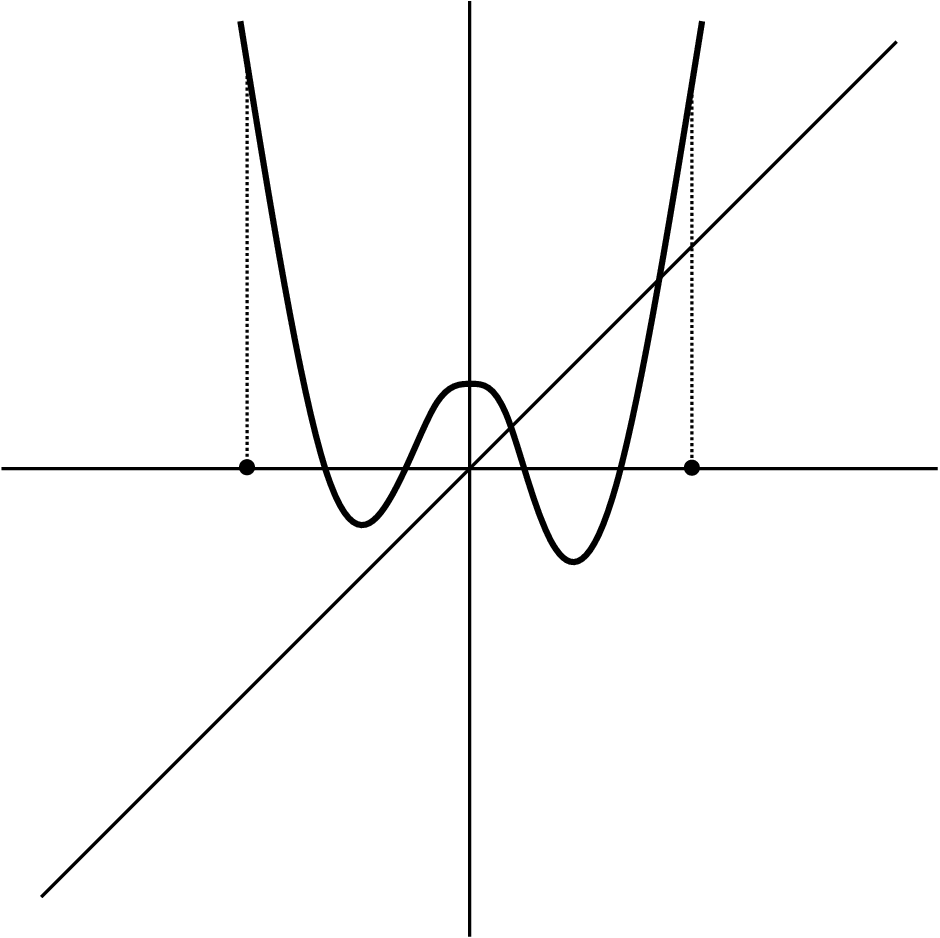} 
				%\begin{overpic}[height=5cm,grid,tics=5]{Fig3.eps} 
					\put(71,43){$Y_0$}
					\put(18,43){$-Y_0$}
				\end{overpic}
				
				Even degree.
			\end{center}
		\end{minipage}
		\begin{minipage}{7cm}
			\begin{center}
				\begin{overpic}[height=5cm]{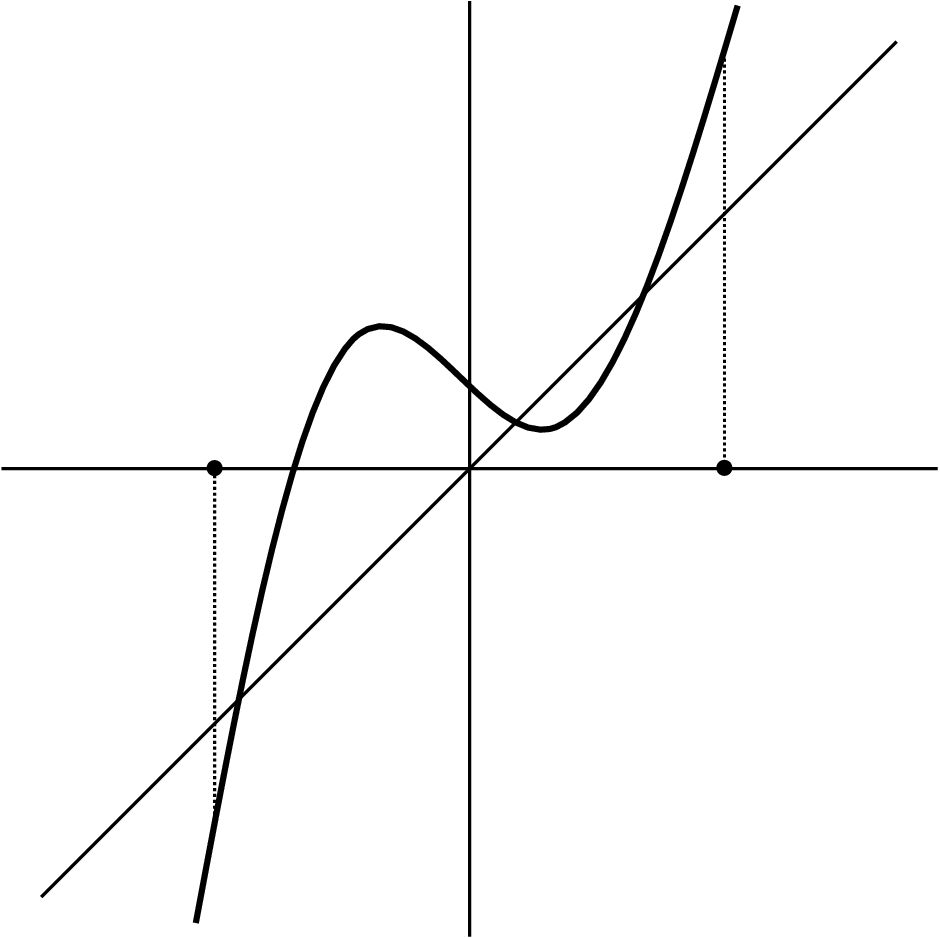} 
				%\begin{overpic}[height=5cm,grid,tics=5]{Fig4.eps} 
					\put(75,43){$Y_0$}
					\put(15,52){$-Y_0$}
				\end{overpic}
				
				Odd degree.
			\end{center}
		\end{minipage}
	\end{center}
\caption{Illustrations of the graph of nonlinear polynomials.}\label{Fig2}
\end{figure}
In particular, it follows that if $y\in\Sigma$ is such that $|y|>Y_0$, then $|P^n(y)|\to\infty$, as $n\to\infty$. Hence, there is a compact $K\subset\mathbb{R}^2$ such that all orbits of $X$ outside $K$ goes to the infinity.

\section{Proof of Theorem~\ref{Main3}}\label{Sec6}

Let $X=(X_1,X_2;\Sigma;\varphi)\in\mathfrak{X}_2$, $X_i=(P_i,Q_i)$, be given by
\begin{equation}\label{21}
	P_1(x,y)=-y, \quad Q_1(x,y)=x, \quad X_2=-X_1, \quad \varphi(y)=-4y(1+y).
\end{equation}
Consider the first return map $P\colon\Sigma\to\Sigma$ associated to $X$. We claim that if $y\in[0,1]$, then $P(y)=f_4^2(y)$, where $f_4(y)=4y(1-y)$ is the logistic map given in Section~\ref{Sec3.2}. Indeed, consider $y_1\in[0,1]$ and similarly to the proof of Theorem~\ref{Main1} observe that
\begin{equation}\label{16}
	y_2=-y_1, \quad y_3=\varphi(y_2), \quad y_4=-y_3, \quad P(y_1)=\varphi(y_4),
\end{equation}
see Figure~\ref{Fig8}.
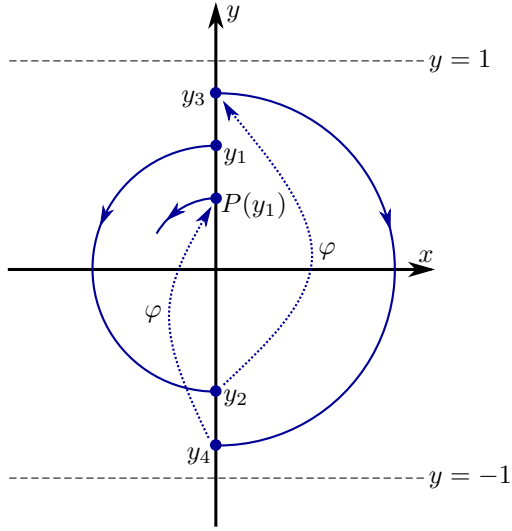
\begin{figure}[h]
	\begin{center}
		\begin{overpic}[height=7cm]{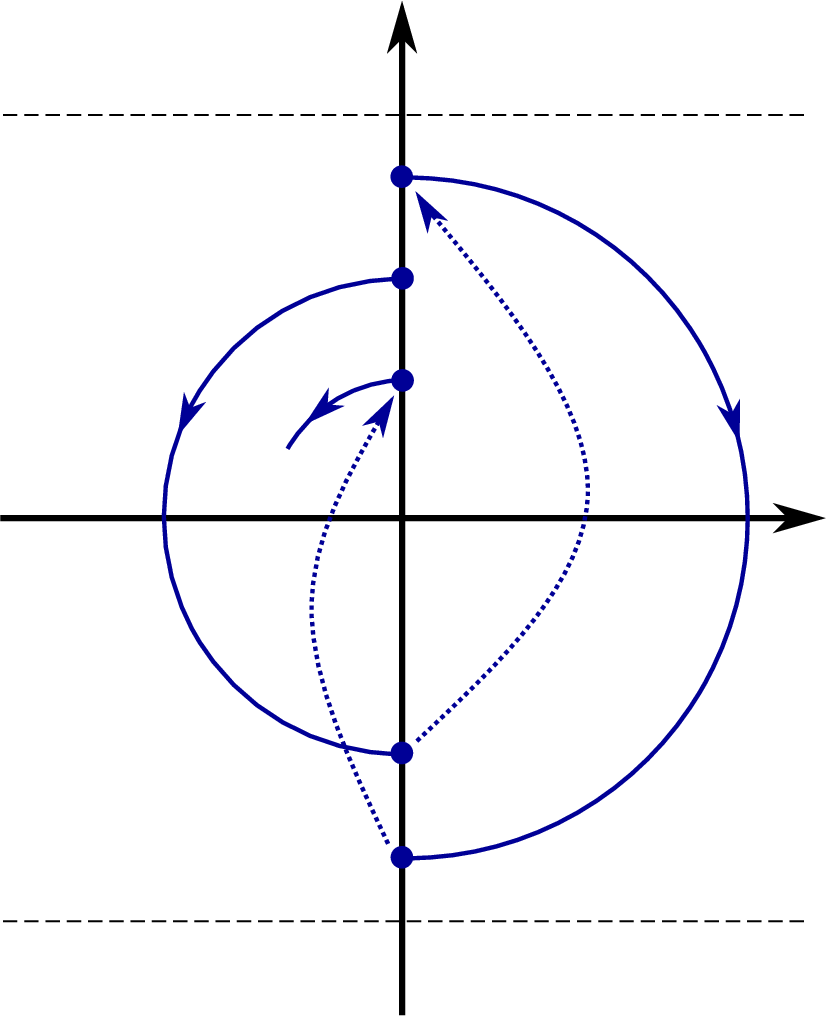} 
		%\begin{overpic}[height=7cm,grid,tics=5]{Fig21x.eps} 
			\put(41,70){$y_1$}
			\put(41,24){$y_2$}	
			\put(33,81){$y_3$}	
			\put(34,13){$y_4$}	
			\put(40.5,60){$P(y_1)$}	
			\put(80,87.5){$y=1$}
			\put(80,8.5){$y=-1$}
			\put(78,50.5){$x$}
			\put(41.5,97){$y$}	
			\put(59,52){$\varphi$}
			\put(26,40){$\varphi$}		
		\end{overpic}
		
	\end{center}
	\caption{Illustration of the first return of $y_1\in[0,1]$.}\label{Fig8}
\end{figure}
Observe also that $\varphi(-y)=f_4(y)$. Hence it follows from \eqref{16} that,
	\[P(y_1)=\varphi(-\varphi(-y_1))=\varphi(-f_4(y_1))=f_4^2(y).\]
This proves the claim. Since $f_4([0,1])=[0,1]$, it follows that $P([0,1])=[0,1]$ and thus $[0,1]$ is invariant by $P$. From \eqref{21} it is clear that if $y\in[0,1]$, then the pieces of orbits of $X_1$ and $X_2$ passing through $(0,y)$ spans the set,
	\[K=\{(x,y)\in\mathbb{R}^2\colon x^2+y^2\leqslant1\}.\]
Hence $K$ is a positive invariant compact set of $X$. Moreover it follows from Proposition~\ref{P5} that $P$ is chaotic on the interval $[0,1]$ and thus it is clear that $X$ is chaotic on $K$.

\section{On the structure of the first return map}\label{Sec7}

Let $X\in\mathfrak{X}_n$. We remind from the proofs of Theorems~\ref{Main1} and \ref{Main2} that its associated first return map $P\colon\Sigma\to\Sigma$ is given by
	\[P(y)=\left\{\begin{array}{ll}
		\varphi_n(\eta_2-\varphi_n(\eta_1-y)) & \text{if } y\in J_1, \vspace{0.2cm} \\
		\varphi_n(\varphi_n(\eta_1-y)) & \text{if } y\in J_2, \vspace{0.2cm} \\
		\varphi_n(\eta_2-\varphi_n(y)) & \text{if } y\in J_3, \vspace{0.2cm} \\
		\varphi_n(\varphi_n(y)) & \text{if } y\in J_4, \vspace{0.2cm} \\
	\end{array}\right.\]
where $J_j\subset\mathbb{R}$ are such that $\cup_j J_j=\mathbb{R}$. In this section we provide a characterization of these subsets.

\begin{proposition}
	Given $n\in\mathbb{N}$, let $X\in\mathfrak{X}_n$ and $P\colon\Sigma\to\Sigma$ be its associated first return map. If $J_j\neq\emptyset$, then it is the union of at most $n$ pairwise disjoint intervals.
\end{proposition}

\begin{proof} Let $\mathcal{O}_i\in\Sigma$ be the unique tangent point between the orbits of $X_i$ and $\Sigma$. Observe that $\Sigma\backslash\{\mathcal{O}_i\}$ has two connected components and that $X_i$ is transversal to $\Sigma$ at them. Moreover in one of them the orbits of $X_i$ enter the region $A_i$, while in the other component the orbits leave it. Let $I_i^+$ and $I_i^-$ denote such components, respectively, see Figure~\ref{Fig11}.	
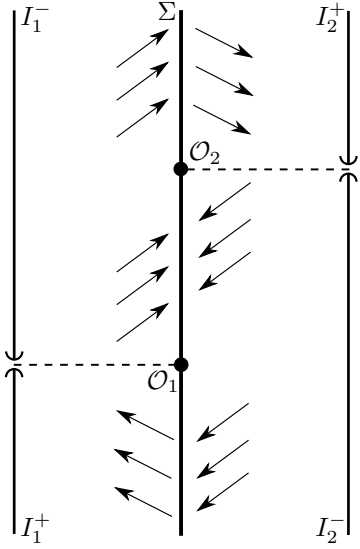
\begin{figure}[h]
	\begin{center}
		\begin{overpic}[height=7cm]{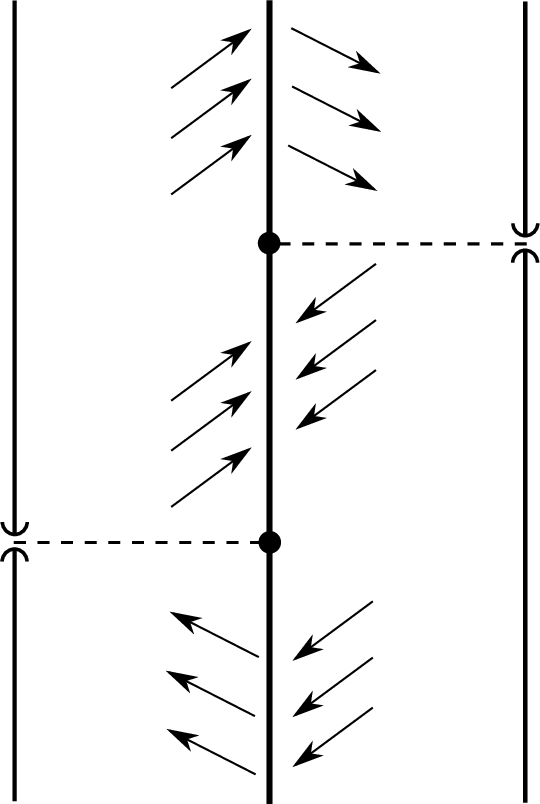} 
		%\begin{overpic}[height=7cm,grid,tics=5]{Fig24.eps} 	
			\put(29,98){$\Sigma$}
			\put(27,28){$\mathcal{O}_1$}
			\put(35,71.5){$\mathcal{O}_2$}
			\put(3,0){$I_1^+$}
			\put(3,97){$I_1^-$}
			\put(59,0){$I_2^-$}
			\put(59,97){$I_2^+$}
		\end{overpic}		
	\end{center}
	\caption{Illustration of the intervals $I_i^\pm$.}\label{Fig11}
\end{figure}
Observe that the orbit of $X_i$ with initial condition $\mathcal{O}_i$ also enters our leaves $A_i$, depending on whether $\mathcal{O}_i$ is a \emph{visible} or \emph{invisible} fold point, respectively. Hence exactly one of the components $I_i^+$ or $I_i^-$ can be extended to $\mathcal{O}_i$, $i\in\{1,2\}$. 

Let $y\in\Sigma$. Observe that $y\in J_1$ (i.e. the orientation of the piece of orbits of its first return agree) if, and only if, $y\in I_1^+$ and $\varphi(\eta_1-y)\in I_2^+$. Hence, if we let $\sigma(y)=\eta_1-y$, it follows that $y\in J_1$ if, and only if, $y\in I_1^+\cap(\varphi\circ\sigma_1)^{-1}(I_2^+)$. Since $I_2^+$ is an interval, $\sigma$ is an affine map and $\varphi$ is a polynomial of degree $n$, it follows that $(\varphi\circ\sigma)^{-1}(I_2^+)$ is a collection of at most $n$ pairwise disjoint intervals and thus 
	\[J_1=I_1^+\cap(\varphi\circ\sigma)^{-1}(I_2^+)\]
is a collection of at most $n$ pairwise disjoint intervals. Similarly, we observe that
	\[J_2=I_1^+\cap(\varphi\circ\sigma)^{-1}(I_2^-), \quad J_3=I_1^-\cap(\varphi\circ\sigma)^{-1}(I_2^+), \quad J_4=I_1^-\cap(\varphi\circ\sigma)^{-1}(I_2^-),\]
and the proof follows similarly. \end{proof}

\begin{remark}
	Observe that the connected components of $J_j$, $j\in\{1,2,3,4\}$, may be open, closed or neither of them, depending on the visibility of the fold points $\mathcal{O}_i$, $i\in\{1,2\}$.
\end{remark}

\section{Conclusion and further thoughts}\label{Sec8}

The main inspiration of this work is the paper of Llibre and Teixeira~\cite{LliTei2018} about Filippov systems formed by two linear centers and having $x=0$ as discontinuity line. One of the main conclusions of the paper is that such systems cannot have limit cycles. Actually, either it does not have periodic orbits or every orbit is periodic. Therefore, its dynamics is relatively simple.

Inspired by this work and the raising notion of hybrid systems, we wondered what could happen if we allow \emph{jumps} on the discontinuity line. As a result, we discovered not only that limit cycles are allowed with arbitrarily small ``perturbations'' in the jump, such as $\varphi(y)=ay$, with $a=1\pm\varepsilon$ and $\varepsilon>0$ small (remind that if $\varphi(y)=y$, then our hybrid systems boils down to the Filippov system studied at \cite{LliTei2018}), but also that such systems allow \emph{chaotic} dynamics. Therefore, we conclude that hybrid systems with simple formulation can have rich dynamics. 

We also observe that a complete characterization of the dynamics of $X\in\mathcal{X}_n$ depends on the characterization of its first return map, which is a piecewise polynomial map on the real line. This, together with the fact that the systems studied here are a generalization of the Filippov systems studied in \cite{LliTei2018}, illustrates that hybrid systems can be seen as a three-fold bridge connecting continuous, piecewise continuous and discrete dynamical systems.

\section*{Acknowledgments}

The first author is partially supported by Agencia Estatal de Investigaci\'on grant PID2022-136613NB-100, AGAUR (generalitat de Catalunha) grant 2021SGR00113 and by the Reial Acadèmia de Cièncias i Arts de Barcelona. The second author is supported by S\~ao Paulo Research Foundation (FAPESP), grants 2019/10269-3, 2021/01799-9 and 2022/14353-1.

\end{document}